\documentclass[12pt,reqno]{amsart}
\usepackage[english]{babel}
\usepackage{amsmath}
\usepackage{amsthm}
\usepackage{anysize}
\usepackage{mathtools}
\usepackage{mathrsfs}
\usepackage{amssymb}
\usepackage{xcolor}
\usepackage{xfrac}
\usepackage{esint}
\usepackage{eucal}
\usepackage{graphicx}
\usepackage{float}
\usepackage{array}
\usepackage{enumitem}
\usepackage{tikz-cd}
\usepackage{centernot}
\usepackage{stmaryrd}
\usepackage{comment}
\excludecomment{confidential}
\usepackage{epstopdf}
\usepackage{mathabx}

\usepackage{graphicx,color}
\usepackage{tikz,pgfplots,pgf}
\usepackage[font=small]{caption}

\newtheorem{theorem}{Theorem}[section]
\newtheorem{lemma}{Lemma}[section]

\newtheorem{corollary}{Corollary}[section]
\theoremstyle{definition}

\theoremstyle{definition}

\newtheorem{definition}{Definition}[section]

\usepackage[english]{babel}
\usepackage{amsmath}
\usepackage{amsthm}
\usepackage{mathtools}
\usepackage{mathrsfs}
\usepackage{amssymb}
\usepackage{xcolor}
\usepackage{xfrac}
\usepackage{esint}
\usepackage{graphicx}
\usepackage{float}
\usepackage{array}
\usepackage{enumitem}
\usepackage[new]{old-arrows}
\usepackage[all,cmtip]{xy}
\usepackage{tikz-cd}
\usepackage{centernot}
\usepackage{stmaryrd}
\usepackage{comment}
\usepackage{bbm}

\newcommand{\mc}{\mathcal}
\newcommand{\mf}{\mathfrak}

\newcommand{\N}{\mathbb{N}}
\newcommand{\C}{\mathbb{C}}

\renewcommand{\l}{\lambda}

\renewcommand{\subsetneq}{\varsubsetneq}

\numberwithin{equation}{section}

\makeatletter
\@namedef{subjclassname@2020}{\textup{2020} Mathematics Subject Classification}
\makeatother

\DeclareMathOperator{\ord}{ord}

\DeclareMathOperator{\Tor}{Tor}

\DeclareMathOperator{\Spec}{Spec}

\def\l{\lambda}

\renewcommand{\subsetneq}{\varsubsetneq}

\theoremstyle{definition}

\theoremstyle{remark}

\DeclareMathOperator{\spec}{Spec}
\DeclareMathOperator{\proj}{Proj}

\begin{document}

\title{Spectral Theory and Bezout Theorem}

\author[J. C. Sampedro]{Juan Carlos Sampedro} \thanks{The author has been supported by the Research Grant PID2021--123343NB-I00 of the Spanish Ministry of Science and Innovation and by the Institute of Interdisciplinar Mathematics of Complutense University.}
\address{Departamento de Matemática Aplicada a la Ingeniería Industrial \\
	E.T.S.D.I. \\
	Ronda de Valencia 3 \\
	Universidad Politécnica de Madrid \\
	Madrid, 28012, Spain.}
\email{juancarlos.sampedro@upm.es}

\subjclass[2020]{}

\begin{abstract}
This note investigates the hidden relationship between the concept of algebraic multiplicity of an eigenvalue and the local intersection index of algebraic varieties. 
\end{abstract}

\maketitle

\section{Introduction}

Given an algebraically closed field $k$ and $n\in\N$, we denote by $\mc{L}(k^{n})$ the vector space consisting on linear maps $T:k^{n}\to k^{n}$, by $GL(k^{n})$ the subset of $\mc{L}(k^{n})$ consisting on linear isomorphisms and by $\mc{S}(k^{n})$ the space of non-invertible linear maps, that is, $\mc{S}(k^{n}):=\mc{L}(k^{n})\setminus GL(k^{n})$. Note that
$
\mc{L}(k^{n})=\spec k[x_{1},\dots, x_{n^{2}}]
$,
after the usual identification of linear maps with matrices. Moreover,
$$
\mc{S}(k^{n})=\spec \left(\frac{k[x_{1},\dots,x_{n^{2}}]}{\mf{D}}\right),
$$
where $\mf{D}:=\langle \det(x_{1},\dots,x_{n^{2}})\rangle \subset k[x_{1},\dots,x_{n^{2}}]$. This scheme is reduced and irreducible, hence integral, of dimension $n^2-1$. Moreover, it is Cohen--Macaulay. 
\par Classical spectral theory deals with straight lines of linear operators of the special form
\begin{equation*}
	L(\l):=\l I_{n}-T, \quad \l\in k,
\end{equation*}
where $T\in \mc{L}(k^{n})$ is a fixed linear operator and $I_{n}$ denotes the identity on $k^{n}$. The classical spectrum of the linear operator $T\in\mc{L}(k^{n})$ is defined by
$\sigma(T):=\{\lambda\in k  :  \det(L(\l))=0\}$.
Since the characteristic polynomial $\det(L(\l))\in k[\l]$ has degree $n$ in $\l$, the set $\sigma(T)$ has at most $n$ elements. If $\sigma(T)=\{\lambda_{1},\lambda_{2},...,\lambda_{m}\}$ with $m\leq n$, then
\begin{equation*}
\det(L(\l))=\prod_{i=1}^{m}(\l-\l_{i})^{\mf{m}[T,\l_{i}]}, \quad \l\in k.
\end{equation*}
The number $\mf{m}[T,\l_i]\in \mathbb{N}$ is known as the algebraic multiplicity of $T$ at the eigenvalue $\l_i$. In other words, for $\l_{0}\in\sigma(T)$, $\mf{m}[T,\l_{0}]:=\ord_{\l=\l_{0}}\det(L(\l))$.
There is an intrinsic way of computing the algebraic multiplicity in terms of the dimension of certain linear subspace of $k^{n}$. Indeed, if $\nu(\l)$ denotes the least positive integer for which
$$
N[(\l I_{n}-T)^{\mu}] = N[(\l I_{n}-T)^{\nu(\l)}], \quad \forall \mu\geq \nu(\l),
$$
by Jordan's theorem, the algebraic multiplicity of the eigenvalue $\lambda_{i}\in\sigma(T)$ can be computed as
$$
	\mf{m}[T,\l_{i}]=\dim \bigcup_{\mu=1}^{\infty}N[(\l_{i} I_{N}-T)^{\mu}]=\dim N[(\l_{i}I_{n}-T)^{\nu(\l_{i})}].
$$
The classical spectral theorem states that 
\begin{equation}
	\label{I1}
\sum_{\l\in\sigma(T)}\mf{m}[T,\l]=n \; \; \text{and} \; \; k^{n}=\bigoplus_{\l\in\sigma(T)} N[(\l I_{n}-T)^{\nu(\l)}].
\end{equation}

A similar spectral theory exists for matrix polynomials of the form
\begin{equation}
	\label{Eq.1}
	L(\l):=\sum_{i=0}^{m}T_{i}\lambda^{i}, \quad \l\in k,
\end{equation}
for some fixed linear operators $T_{i}\in \mathcal{L}(k^{n})$, $1\leq i\leq m$, where $T_{m}\neq 0$. We denote by $\mc{P}(k^{n})$ the space of polynomial operator curves of the form \eqref{Eq.1}. The spectrum of any $L\in\mc{P}(k^{n})$ is naturally given by
\begin{equation*}
	\Sigma(L)=\{\l\in k: \det(L(\l))=0\}.
\end{equation*}
Since $\det(L(\l))\in k[\l]$ is a polynomial of degree at most $mn$, the set $\Sigma(L)$ is empty, finite or $\Sigma(L)=k$.
It is easily seen that if $T_{m}\in GL(k^{n})$ or $T_{0}\in GL(k^{n})$, then $\Sigma(L)$ is finite.
Naturally, we define the \textit{algebraic multiplicity} of $L\in\mc{P}(k^n)$ at $\lambda_{0}\in\Sigma(\mf{L})$ by
\begin{equation}
	\label{DEFPOL}
	\mf{m}[L,\l_{0}]:=\ord_{\l=\l_{0}}\det(L(\l)).
\end{equation}
It should be observed that $\mf{m}[L,\l_{0}]\in \N\cup\{\infty\}$ and that it is finite if $\Sigma(L)$ is. This concept has been widely studied in several contexts as the monograph of López-Gómez and Mora-Corral \cite{LM} shows.

There is also a spectral theory for monic matrix polynomials, that is, for $L\in \mc{P}(k^{n})$ such that $T_{m}=I_{n}$. This theory, attributable to Gohberg, Lancaster and Rodman \cite{GLR}, has been further developed by several authors \cite{GLR2,GLR3,GLRb,LM,LS21,R} and references therein. For monic polynomials, we have a generalization of identity \eqref{I1}:
\begin{equation}
	\label{I2}
	\sum_{\l\in\Sigma(L)}\mf{m}[L,\l]=nm.
\end{equation}

The aim of this note is to give a geometric interpretation of the algebraic multiplicity for polynomial matrices $L\in\mc{P}(k^{n})$ via intersection theory and to show that identities \eqref{I1}--\eqref{I2} are direct consequences of the celebrated Bezout theorem. 

This interpretation is given through the concept of local intersection multiplicity given by Serre \cite[Ch. V, \S C.1]{Se}. For any given pair $X_1, X_2$ of subvarieties of a smooth variety $X$ with proper intersection and any irreducible component, $C$,  of $X_1\cap X_2$, the local intersection index of $X_1$ and $X_2$ along $C$ is given through
$$
i(X_1,X_2;C)=\sum_{i=0}^{\dim X}(-1)^{i}\ell_{\mathcal{O}_{c,X}} \Tor^{\mathcal{O}_{c,X}}_{i}(\mathcal{O}_{c,X}/\mathfrak{P}_{X_1},\mathcal{O}_{c,X}/\mathfrak{P}_{X_2})
$$
where $\mathcal{O}_{c,X}$ is the local ring of $c\in C$ in $X$, and $\mathfrak{P}_{X_1}$ and $\mathfrak{P}_{X_2}$ are the ideals  of $X_1$ and $X_2$, respectively,  in the ring $\mathcal{O}_{c,X}$. As in our context  $X$ is a smooth algebraic variety over an algebraically closed field and $X_1, X_2$
are two irreducible Cohen--Macaulay subvarieties of $X$ such that $X_1\cap X_2=\{x_{1},\dots,x_{m}\}$, with proper intersection, according to, e.g., Eisenbud and Harris  \cite[p. 48]{EiHa}, the local intersection index of $X_1$ and $X_2$ at $x_{j}$ reduces to
\begin{equation}
	\label{1.1.12}
	i(X_1,X_2;x_{j})=\ell_{\mathcal{O}_{x_{j},X}}(\mathcal{O}_{x_{j},X}/(\mathfrak{P}_{X_1}+\mathfrak{P}_{X_2})).
\end{equation}
Under this concept, Bezout theorem for curves and hypersurfaces can be stated as follows (see, for instance, Hartshorne \cite[Th.7.7, Ch.1]{Ha} or Fulton \cite[\S 8.4]{Fu98}). Let $X$ be an algebraic curve in $\mathbb{P}^{n}(k)$ and $H$ an algebraic hypersurface in $\mathbb{P}^{n}(k)$ not containing $X$. If $X\cap H=\{z_1,\dots,z_{n}\}$, then
\begin{equation}
	\label{BT}
\sum_{i=1}^{n}i(X, H; z_{i}) = \deg(X) \cdot \deg(H),
\end{equation}
where $\deg(Y)$ stands for the degree of the algebraic variety $Y\subset \mathbb{P}^{n}(k)$.

The main result of this note states that under reasonable assumptions on $L\in\mc{P}(k^{n})$, the algebraic multiplicity of $L$ at $\l\in\Sigma(L)$ can be reinterpreted as the intersection multiplicity of the hypersurface $\mc{S}(k^{n})$ with the curve $L(k)$. That is, for $\l \in \Sigma(L)$,
\begin{equation}
	\label{p}
	\mf{m}[L,\l] = i(\mc{S}(k^{n}), L(k); L(\l)).
\end{equation}
Moreover, by \eqref{BT}, if $L$ does not meet $\mc{S}(k^{n})$ at infinity, then
\begin{equation}
	\label{q}
	\sum_{\l\in\Sigma(L)}\mf{m}[L,\l] = \sum_{\l\in\Sigma(L)} i(\mc{S}(k^{n}), L(k); L(\l)) = \deg(L)\cdot \deg(\mc{S}(k^{n})) = mn,
\end{equation}
recovering \eqref{I2}.

The article is organized as follows. In Section \ref{S2} we introduce the notion of admissible curve and give a precise formulation of identities \eqref{p}--\eqref{q} in this context. We also show that such interpretation fails in general if admissibility is not assumed. In Section \ref{S3}, we provide a geometric, albeit less direct, interpretation for general polynomial curves $L\in\mc{P}(k^{n})$, and prove analogous versions of identities \eqref{p}--\eqref{q}. Finally, Section \ref{S4} presents a worked example illustrating these notions.

\section{Results for admissible curves}\label{S2}

\par We will prove that the concept of algebraic multiplicity for polynomial operator curves is related to the intersection index of algebraic varieties. However, we have to restrict, at a first stage, to the subclass of admissible curves.

\begin{definition}
	A polynomial operator curve $L(\l):=(L_{i}(\l))_{1\leq i \leq n^{2}}\in \mc{P}(k^{n})$ is said to be \textit{admissible} if 
	$$
		k[L_1(\l),\cdots,L_{n^{2}}(\l)]=k[\l].
	$$
	The set of admissible operator curves will be denoted by $\mc{A}(k^{n})$.
\end{definition}

For any given $L\in \mc{A}(k^{n})$, the associated morphism of schemes $L:\Spec k[\l] \to \Spec k[x_{1},\dots,x_{n^{2}}]$ induce an isomorphism onto its image. Therefore, $L(k)$ is a scheme isomorphic to the affine line $\Spec k[\l]$. Consequently, $\text{Im}(L)$ is smooth and $L$ is injective. Moreover, 
$$
L(k)\simeq \Spec \left(\frac{k[x_{1},\dots,x_{n^{2}}]}{\mf{I}}\right), \quad \mf{I}:=\ker \varphi,
$$
where $\varphi: k[x_1,\dots,x_{n^{2}}]\to k[\l]$ is the ring homomorphism associated to $L$. 

The next lemma will be useful in the sequel.
\begin{lemma}
	\label{L2.1}
	Let $L\in \mc{A}(k^{n})$ and $\l_{0}\in k$. Then, the following ideal identity holds
	$$
	\langle L_{1}(\l)-L_{1}(\l_{0}), \cdots, L_{n^{2}}(\l)-L_{n^{2}}(\l_{0})\rangle = \langle (\l-\l_0) \rangle \quad \text{in} \; \; k[\l].
	$$
\end{lemma}

\begin{proof}
	For each $1\leq i \leq n^{2}$, there exist $k_{i}\geq 1$ and $f_{i}\in k[\l]$ with $f_{i}(\l_0)\neq 0$, such that
	$$
	L_{i}(\l)-L_{i}(\l_{0})=(\l-\l_{0})^{k_{i}}f_{i}(\l).
	$$
	By the injectivity of $L$, there does not exist $\l_{1}\neq \l_{0}$ such that $f_{i}(\l_{1})=0$ for every $1\leq i \leq n^{2}$. Therefore, 
	\begin{align*}
		&\langle L_{1}(\l)-L_{1}(\l_{0}), \cdots, L_{n^{2}}(\l)-L_{n^{2}}(\l_{0})\rangle  = \langle (\l-\l_0)^{k_{1}}f_{1}(\l), \cdots, (\l-\l_0)^{k_{n^{2}}}f_{n^{2}}(\l)\rangle \\
		&=\langle \text{gcd}\left[(\l-\l_0)^{k_{1}}f_{1}(\l), \cdots, (\l-\l_0)^{k_{n^{2}}}f_{n^{2}}(\l)\right]\rangle 
		=\langle (\l-\l_{0})^{k_{-}}\rangle,
	\end{align*}
	where $k_{-}:=\text{min}\{k_{1},\cdots, k_{n^{2}}\}$. Suppose by the way of contradiction that $k_{-}\geq 2$. Then, 
	$$
	L_{i}(\l)=(\l-\l_0)^{k_{i}}f_{i}(\l)+L_{i}(\l_{0}), \quad k_{i}\geq 2, \quad 1\leq i \leq n^{2}
	$$
	and therefore
	$$
	k[L_{1}(\l),\dots,L_{n^{2}}(\l)]=k[(\l-\l_0)^{k_{1}}f_{1}(\l), \dots, (\l-\l_0)^{k_{n^{2}}}f_{n^{2}}(\l)].
	$$
	A degree comparison argument shows that $k[L_{1}(\l),\dots,L_{n^{2}}(\l)]\neq k[\l-\l_0]=k[\l]$. Consequently $k_{-}=1$ and the proof is complete.
\end{proof}

The next result proves that for admissible curves $L\in\mc{A}(k^{n})$, the algebraic multiplicity is, in fact, the intersection multiplicity of the hypersurface $\mc{S}(k^{n})$ and the curve $L(k)$ at $L(\l_{0})$.

\begin{theorem}
	\label{Th2.1}
	Let $L\in \mathcal{A}(k^{n})$ such that $\Sigma(L)\neq k$. Then, for every $\l_{0}\in\Sigma(L)$, it holds that
	\begin{equation}
		\label{Eq}
	\mf{m}[L,\l_{0}]=i(\mathcal{S}(k^{n}), L(k); L(\l_{0})).
	\end{equation}
\end{theorem}

\begin{proof}
	The ring homomorphism
	$$
	\varphi: k[x_{1},\cdots,x_{n^{2}}] \longrightarrow k[\l], \quad x_{i}\mapsto L_{i}(\l),
	$$
	where $L(\l)\equiv (L_{i}(\l))_{1\leq i \leq n^{2}}$, induce an isomorphism
	$$
	\phi: \frac{k[x_{1},\cdots,x_{n^{2}}]}{\mf{I}} \longrightarrow k[L_{1}(\l), \cdots, L_{n^{2}}(\l)]= k[\l].
	$$
	Let $\mf{p}=\langle L_{i}(\l)-L_{i}(\l_0) : i=1,\dots, n^{2}\rangle$ and  $\mf{q}=\langle (\l-\l_0)\rangle \subset k[\l]$, then by Lemma \ref{L2.1}, $\phi(\mf{p})=\mf{q}$.
	Consequently, the induced morphism $\bar{\phi}:\frac{\mc{O}_{\mf{p}}}{\mf{I}}\to \mc{O}_{\mf{q}}$
	is an isomorphism, where $\mathcal{O}_{\mf{p}}$ stands for  the local ring of $k[x_1,\dots,x_{n^{2}}]$ at $\mf{p}$ and $\mc{O}_{\mf{q}}$ is the local ring of $k[\l]$ at $\mf{q}$.
	
	On the other hand, denoting $\chi:=\mf{m}[L,\l_{0}]$, there exists $f\in k[\l]$, $f(\l_{0})\neq 0$, such that
	$$
	\det(L(\l))=(\l-\l_0)^{\chi}f(\l), \quad \l\in k.
	$$
	Then,
	$
	\bar{\phi}(\det(x_{1},\cdots,x_{n^{2}}))=(\l-\l_0)^{\chi}f(\l)
	$
	and $\bar{\phi}$ induces the isomorphism
	\begin{equation}
		\label{iso}
		\frac{\mc{O}_{\mf{p}}}{\mf{D}+\mf{I}} \simeq \frac{\mc{O}_{\mf{q}}}{\langle (\l-\l_{0})^{\chi} \rangle},
	\end{equation}
	where, recall that, $\mf{D}:=\langle \det(x_1,\dots,x_{n^{2}})\rangle$. Since $\mc{S}(k^{n})$ and $L(k)$ are Cohen--Macaulay,  the local
	intersection multiplicity of $\mc{S}(k^{n})$ and $L(k)$ at $L(\l_0)$ is given by
	\begin{align*}
		i(\mc{S}(k^{n}),L(k);L_0) &=
		\ell_{\mathcal{O}_{\mf{p}}}\big(\mathcal{O}_{\mf{p}}/(\mf{D}+\mf{I})\big) =\dim_{k}\big(\mathcal{O}_{\mf{p}}/(\mf{D}+\mf{I})\big) \\ &=\dim_{k}\left(\frac{\mc{O}_{\mf{q}}}{\langle (\l-\l_{0})^{\chi} \rangle}\right)=\chi.
	\end{align*}
	This concludes the proof.
\end{proof}

Noting that polynomial curves of degree one, that is, operator lines $L(\l)=\l T_{1}+ T_{0}$, $T_{1}\neq 0$, are always admissible, we obtain the following direct consequence of Theorem \ref{Th2.1}.

\begin{corollary}
Let $T_0,T_1\in\mathcal{L}(k^n)$ with $T_{1}\neq 0$, and  $L(\lambda)=\lambda T_1+T_0$ such that $\Sigma(L)\neq k $. Then, for every $\l_0\in \Sigma(L)$, it holds that
	\begin{equation*}
		\mf{m}[L,\lambda_{0}]=
		i(\mc{S}(k^n),L(k); L(\l_0)).
	\end{equation*}
\end{corollary}

	Identity \eqref{Eq} may be false without the admissibility condition. In fact, consider $L\in\mc{P}(k^{2})$ given by
	$$
	L(\l):=\left[
	\begin{array}{ll}
		\l^{2} & 0 \\
		0 & 1
	\end{array}
	\right], \quad \l\in k.
	$$
	Then, $\mf{m}[L,0]=2$ but $i(\mc{S}(k^{2}),L(k); L(0))=1$. This is because $k[\l^{2}]\neq k[\l]$ and hence $L\notin \mc{A}(k^{2})$.

In order to prove identity \eqref{q} via Bezout theorem, we have to work in the projective space. Consider the projectivization of $\mc{L}(k^n)$ given by  $\mathbb{P}\mc{L}(k^n):=\proj k[x_0,x_1,\dots,x_{n^{2}}]$. Since the polynomial $\det(x_1,\dots,x_{n^{2}})\in k[x_1,\dots,x_{n^{2}}]$ is homogeneous, the projective closure of $\mc{S}(k^{n})$ is given by
$$
\bar{\mc{S}}(k^{n})=\proj\left(\frac{k[x_0,x_1,\dots,x_{n^{2}}]}{\mf{D}}\right), \quad \mf{D}:=\langle \det(x_1,\dots,x_{n^{2}})\rangle.
$$
Let $L\in \mc{P}(k^{n})$. Then, the projective closure of $\text{Im}(L)$ is the image of 
$$
\bar{L}:\mathbb{P}^{1}(k)\longrightarrow \mathbb{P}\mc{L}(k^{n}), \quad [\l : \mu] \mapsto [\mu^{d} : \mu^{d}L_{1}\left(\frac{\l}{\mu}\right) : \dots : \mu^{d}L_{n^{2}}\left(\frac{\l}{\mu}\right)],
$$
where $d:=\max_{i}\deg(L_{i}(\l))$. That is, $\overline{\text{Im}(L)}=\text{Im}(\bar{L})$. We define the \textit{projective spectrum} of $L\in \mc{P}(k^n)$ as $\bar{\Sigma}(L)=\Sigma(L)\cup \Sigma_{\infty}(L)$, where
$$
\Sigma_{\infty}(L):=\left\{
\begin{array}{ll}
	\emptyset & \text{if} \; \det(\bar{L}([\l:0]))\neq 0, \\
	\{\infty\} & \text{if} \; \det(\bar{L}([\l:0]))=0.
\end{array}
\right.
$$
Philosophically, we include the singleton $\{\infty\}$ to $\Sigma(L)$ if $L$ intersects $\mc{S}(k^{n})$ at infinity.

\begin{theorem}
	\label{Th2.2}
	Let $L\in\mc{A}(k^{n})$ such that $\Sigma(L)=\bar{\Sigma}(L)\neq k$. Then,
	\begin{equation}
		\label{Eq1}
	\sum_{\l\in\Sigma(L)}\mf{m}[L,\l] = \sum_{\l\in\Sigma(L)} i(\mc{S}(k^{n}), L(k); L(\l)) = mn,
	\end{equation}
	where $m:=\max_{i}\deg(L_{i}(\l))$.
\end{theorem}

\begin{proof}
	Since $\Sigma(L)=\bar{\Sigma}(L)$ and $\deg(\bar{\mc{S}}(k^{n}))=n$, Bezout theorem applied to our particular case states that 
	$$
	\sum_{\l\in \Sigma(L)} i(\bar{\mc{S}}(k^{n}), {\hbox{Im}}(\bar{L}); \bar{L}(\l) ) = n\cdot \deg({\hbox{Im}}(\bar{L})).
	$$
	As $\Sigma(L)=\bar{\Sigma}(L)$, we can work in the affine chart $\mc{U}_{0}:=\left\{[1:x_1:\dots:x_{n^{2}}] : x_{i}\in k\right\}\subset \mathbb{P}\mc{L}(k^{n})$. The isomorphism
	$$
	\varphi_{0}: \mc{U}_{0}\subset \mathbb{P}\mc{L}(k^{n})\longrightarrow \Spec k[x_{1},\dots,x_{n^{2}}], \quad [1:x_{1}:\dots:x_{n^{2}}]\mapsto (x_{1},\dots,x_{n^{2}}),
	$$
	induce an isomorphism at the level of local rings. Hence, 
	$$ i(\bar{\mc{S}}(k^{n}), {\hbox{Im}}(\bar{L}); \bar{L}(\l) ) = i(\mc{S}(k^{n}), L(k); L(\l)), \quad  \l\in\Sigma(L). $$
	Finally, since $L: \Spec k[\l] \to \text{Im}(L)$ is an isomorphism, $\bar{L}:\mathbb{P}^{1}(k)\to \text{Im}(\bar{L})$ is bijective and therefore, by directly applying the definition of degree of a projective variety, we deduce that $\deg({\hbox{Im}}(\bar{L}))=\max_{i}\deg(\bar{L}_{i})=\max_{i}\deg(L_{i})$. This concludes the proof.
\end{proof}

If $L(\l)=\l T_{1}+T_{0}$ with $T_{1}\in GL(k^{n})$, then a direct computation shows that $L(k)$ does not intersect $\mc{S}(k^{n})$ at infinity and consequently, that $\Sigma(L)$ is finite. This observation leads to the following direct consequence of Theorem \ref{Th2.2}.

\begin{corollary}
	Let $T_0\in\mathcal{L}(k^n)$, $T_{1}\in GL(k^{n})$ and  $L(\lambda)=\lambda T_1+T_0$. Then, 
	$$
	\sum_{\l\in\Sigma(L)}\mf{m}[L,\l] = \sum_{\l\in\Sigma(L)} i(\mc{S}(k^{n}), L(k); L(\l)) = n.
	$$
\end{corollary}

\section{Results in the general case}
\label{S3}

In this section, we provide a geometric interpretation of the algebraic multiplicity for general polynomial curves, without assuming the admissibility condition. To this end, we introduce the notion of a \emph{resolution} of a given curve.

Let $L\in\mc{P}(k^{n})$. We define the \emph{resolution} of $L$ as the scheme
\[
\mc{R}(L):=\Spec \left(\frac{k[\l, x_1,\dots,x_{n^{2}}]}{\mf{P}_{L}}\right), \quad \text{where} \quad \mf{P}_{L}:=\langle x_{1}-L_{1}(\l),\dots, x_{n^{2}}-L_{n^{2}}(\l)\rangle.
\]
The scheme $\mc{R}(L)$ is smooth and can be regarded as a desingularization of the image of $L$. We also consider the natural projection
\[
\pi: \Spec k[\l, x_{1},\dots,x_{n^{2}}] \longrightarrow \Spec k[x_{1},\dots, x_{n^{2}}],
\]
which forgets the parameter $\l$. A straightforward computation shows that
\[
\Sigma(L)=\{\l\in k : (T,\l)\in \mc{R}(L)\cap \pi^{-1}(\mc{S}(k^{n})) \text{ for some } T\in\mc{L}(k^{n})\}.
\]
The following result shows that the algebraic multiplicity of a general curve $L\in\mc{P}(k^{n})$ at a given eigenvalue $\lambda$ can be interpreted as the intersection multiplicity of the resolution $\mc{R}(L)$ with the scheme $\pi^{-1}(\mc{S}(k^{n}))$ at the point $(\lambda, L(\lambda))$.

\begin{theorem}
	Let $L\in \mc{P}(k^{n})$ such that $\Sigma(L)\neq k$. Then, for every $\l_0\in\Sigma(L)$, it holds that
	$$
	\mf{m}[L,\l_{0}]= i(\pi^{-1}(\mc{S}(k^{n})), \mc{R}(L); c_{\l_{0}}), \quad c_{\l_{0}}:=(L(\l_{0}),\l_{0}).
	$$
\end{theorem}
\begin{proof}
	Since $\pi^{-1}(\mc{S}(k^{n}))$ and $\mc{R}(L)$ are Cohen--Macaulay schemes, we can compute the local intersection multiplicity as
	$$
	i(\pi^{-1}(\mc{S}(k^{n})), \mc{R}(L); c_{\l_{0}})= \ell_{\mc{O}_{\mf{p}}}\left(\frac{\mc{O}_{\mf{p}}}{ \mf{D}+\mf{P}_{L}}\right),
	$$
	where $\mc{O}_{\mf{p}}$ stands for the local ring of $k[\l,x_{1},\dots,x_{n^{2}}]$ at the prime ideal $\mf{p}:=\langle \l-\l_{0}, x_{i}-L_{i}(\l_{0}) : i=1,\dots, n^{2} \rangle$. A simple computation shows that
	\begin{align*}
	\frac{\mc{O}_{\mf{p}}}{ \mf{D}+\mf{P}_{L}} & = \frac{\mc{O}_{\mf{p}}}{ \langle \det(L(\l)), x_{i}-L_{i}(\l) : i=1,\dots, n^{2}\rangle} \\
	&= \frac{\mc{O}_{\mf{p}}}{ \langle (\l-\l_{0})^{\chi}, x_{i}-L_{i}(\l) : i=1,\dots, n^{2}\rangle},
	\end{align*}
where $\chi:= \mf{m}[L,\l_{0}]$. Finally, 
	\begin{align*}
i(\pi^{-1}(\mc{S}(k^{n})), \mc{R}(L); c_{\l_{0}}) & = \ell_{\mc{O}_{\mf{p}}}\left(\frac{\mc{O}_{\mf{p}}}{ \langle (\l-\l_{0})^{\chi}, x_{i}-L_{i}(\l) : i=1,\dots, n^{2}\rangle}\right) \\ &= \dim_{k}\left(\frac{\mc{O}_{\mf{p}}}{ \langle (\l-\l_{0})^{\chi}, x_{i}-L_{i}(\l) : i=1,\dots, n^{2}\rangle}\right)=\chi.
\end{align*}
This concludes the proof.
\end{proof}

Finally, by adapting the projectivization method from Section~\ref{S2}, Theorem~\ref{Th2.2}, to the context of $\proj k[\l,x_0,x_1,\dots,x_{n^{2}}]$, we obtain a direct formulation of Bézout’s theorem in the general case.

\begin{theorem}
	\label{Th3.2}
	Let $L\in\mc{P}(k^{n})$ such that $\Sigma(L)=\bar{\Sigma}(L)\neq k$. Then,
	$$
	\sum_{\l\in\Sigma(L)}\mf{m}[L,\l] = \sum_{\l\in\Sigma(L)} i(\pi^{-1}(\mc{S}(k^{n})), \mc{R}(L); (L(\l),\l)) = dn,
	$$
	where $d:=\max_{i}\deg(L_{i}(\l))$.
\end{theorem}

In the monic case, that is, when $L\in \mc{P}(k^{n})$ satisfies $T_{m} = I_{n}$, we have $\Sigma(L) = \bar{\Sigma}(L) \neq k$, and hence Theorem~\ref{Th3.2} recovers identity~\eqref{I2}.

\section{Worked example}\label{S4}

In this final section, we present an explicit example to illustrate the main concepts and results developed throughout the paper.

Let $k = \mathbb{C}$ and consider the matrix polynomial $L \in \mc{P}(\mathbb{C}^{2})$ defined by
\[
L(\l):=\begin{bmatrix}
	\l^{2}-1 & \l \\
	\l & 0
\end{bmatrix}
= \begin{bmatrix}
	1 & 0 \\
	0 & 0
\end{bmatrix} \l^{2} + \begin{bmatrix}
	0 & 1 \\
	1 & 0
\end{bmatrix} \l + \begin{bmatrix}
	-1 & 0 \\
	0 & 0
\end{bmatrix}, \quad \l\in \mathbb{C}.
\]
Since $\mathbb{C}[\l^{2}-1,\l] = \mathbb{C}[\l]$, the curve $L$ is admissible, i.e., $L \in \mc{A}(\mathbb{C}^{2})$. A direct computation shows that
\[
\det(L(\l)) = -\l^{2},
\]
so that $\Sigma(L) = \{0\}$ and the algebraic multiplicity at the origin is $\mf{m}[L,0] = 2$.

On the other hand, the image of $L$ can be algebraically described as
\[
L(\mathbb{C}) = \Spec\left( \frac{\mathbb{C}[x,y,z,t]}{\langle x - y^2 + 1,\ y - z,\ t \rangle} \right).
\]
Let $\mathcal{O}_{\mf{p}}$ be the local ring of $\mathbb{C}[x,y,z,t]$ at the point $\mf{p} = \langle x+1, y, z, t \rangle$, which corresponds to $L(0)$. A straightforward computation yields
\begin{align*}
	i(\mc{S}(\mathbb{C}^{2}), L(\mathbb{C}); L(0)) &= \ell_{\mathcal{O}_{\mf{p}}} \left( \frac{\mathcal{O}_{\mf{p}}}{\langle xt - yz,\ x - y^2 + 1,\ y - z,\ t \rangle} \right) \\
	&= \dim_{\mathbb{C}} \left( \frac{\mathcal{O}_{\mf{p}}}{\langle y^2,\ x+1,\ y - z,\ t \rangle} \right) = \dim_{\mathbb{C}} \left( \frac{\mathbb{C}[x,y,z,t]}{\langle y^2,\ x+1,\ y - z,\ t \rangle} \right) = 2.
\end{align*}
Therefore, as predicted by Theorem~\ref{Th2.1}, we recover the equality
\[
\mf{m}[L,0] = i(\mc{S}(\mathbb{C}^{2}), L(\mathbb{C}); L(0)).
\]

However, the identity in Theorem~\ref{Th2.2} does not hold in this case. Indeed, although
\[
i(\mc{S}(\mathbb{C}^{2}), L(\mathbb{C}); L(0)) = 2 = \sum_{\lambda \in \Sigma(L)} i(\mc{S}(\mathbb{C}^{2}), L(\mathbb{C}); L(\lambda)),
\]
we note that
\[
dn = \deg(L) \cdot \deg(\mc{S}(\mathbb{C}^{2})) = 2 \cdot 2 = 4 \neq 2.
\]
This occurs because the curve $L$ intersects $\mc{S}(\mathbb{C}^{2})$ at infinity; in particular, $\Sigma(L) \subsetneq \bar{\Sigma}(L)$. 
To see this explicitly, consider the projective closure of $\operatorname{Im}(L)$, which is the image of the map
\[
\bar{L}: \mathbb{P}^{1}(\mathbb{C}) \longrightarrow \mathbb{P} \mc{L}(\mathbb{C}^{2}), \quad [\lambda : \mu] \mapsto \left[ \mu^{2} : \begin{bmatrix}
	\lambda^{2} - \mu^{2} & \mu\lambda \\
	\mu\lambda & 0
\end{bmatrix} \right].
\]
In particular, when $\mu = 0$, we have
\[
\bar{L}([\lambda : 0]) = \left[ 0 : \begin{bmatrix}
	1 & 0 \\
	0 & 0
\end{bmatrix} \right] \in \bar{\mc{S}}(\mathbb{C}^{2}),
\]
so that $\bar{L}$ meets the closure of the $\mc{S}(\C^{2})$ at the point at infinity. Consequently,
\[
\bar{\Sigma}(L) = \{0,\infty\},
\]
and the missing contribution in identity~\eqref{Eq1} comes from the intersection at infinity:
\[
i(\bar{\mc{S}}(\mathbb{C}^{2}), \bar{L}(\mathbb{P}^{1}); \bar{L}([\lambda:0])) = 2.
\]

\end{document}